\documentclass[12pt,twoside]{amsart}
\usepackage{amsmath}
\usepackage{amssymb}
\usepackage{wasysym}
\usepackage{psfrag}
\usepackage{graphicx,epsf,amsmath}  
\usepackage{epsf,graphicx}
\usepackage{comment}

\newtheorem{thm}{Theorem}[section]
\newtheorem{theorem}{Theorem}[section]
\newtheorem{pro}[thm]{Proposition}
\newtheorem{remark}[thm]{Remark}

\newtheorem{lemma}[thm]{Lemma}

\newtheorem{defi}[thm]{Definition}

\title[]{Maximization of higher order eigenvalues and applications}
\author[n. nadirashvili]{nikolai nadirashvili}
\author[Y. Sire]{Yannick SIRE}

\begin{document}
\maketitle
\begin{abstract}
The present paper is a follow up of our paper \cite{nS}. We investigate here the maximization of higher order eigenvalues in a conformal class on a smooth compact boundaryless Riemannian surface. Contrary to the case of the first nontrivial eigenvalue as shown in \cite{nS}, bubbling phenomena appear. 
\end{abstract}

\tableofcontents
\section{Introduction}

Let $(M,g)$ be a smooth connected compact Riemannian surface without boundary. In this paper, we extremalize higher order eigenvalues in a suitably defined conformal class. If we denote $-\Delta_g$ the Laplace-Beltrami operator on $(M,g)$, the spectrum of $-\Delta_g$ consists of the sequence $\left \{ \lambda_k(g) \right \}_{k \geq 0}$ and satisfies 
$$\lambda_0(g)=0 <\lambda_1(g) \leq \lambda_2(g) \leq ... \leq \lambda_k(g) \leq ...$$

If we assume that the area $A_g(M)$ of $M$ with respect to the metric $g$ is normalized by one then by the fundamental result of Korevaar (see \cite{korevaar} and also \cite{YangYau}), it follows that every $\lambda_k(g)$ for a given $k \geq 0$ has a universal bound depending on the topological type of $M$, over all the metrics $g$ with normalized area. 

The contributions of the present paper are then related to the following extremal problem 
$$\Lambda_k(M)=\sup_{g} \lambda_k(g) A_g(M)$$
where the supremum is taken over all smooth Riemannian metrics $g$ on the manifold $M$. In \cite{nS}, we investigated this problem for $k=1$, i.e. the first nontrivial eigenvalue. Here we address higher order eigenvalues $k \geq 2$.  

We denote by $\lambda_k(g)$ the $k-$eigenvalue of $-\Delta_g$ and we have by the Courant-Hilbert formulas
$$\lambda_k(g)=\max_{U,\, dim(U)=k}\inf_{ u \in U} R_{M,g}(u)$$

where $R_{M,g}(u)$ is the so-called Rayleigh quotient given by 
$$R_{M,g}(u)=\frac{\int_M |\nabla u|^2 dA_g}{\int_M u^2 dA_g}$$

and the infimum is taken over the space 
$$U \subset \left \{ u \in H^1(M),\,\,\,\,\int_M u=0 \right \}. $$

In the previous definition of higher order eigenvalues, the metric is assumed to be smooth. In our case, we will see that the extremal metric is {\sl not} smooth. However, we will construct this metric as a limit of suitably smooth metrics and the associated higher order extremal eigenvalue will be defined in a natural way out of this limit. 

As in our paper \cite{nS}, we define conformal metrics $\bar g$ (belonging to the conformal class denoted $[g]$) as metrics of the form $\bar g=\mu g$ where $\mu: M \to \mathbb R^+$  is an $L^1$ function on $M$ with mass $1$, i.e. a probability density.

 We then define  
$$\tilde \Lambda_k(M, [g])=\sup_{\tilde g \in [g]} \lambda_k(\tilde g). $$

%
%
%

We prove the following result. 

\begin{theorem}\label{main1}
Let $(M,g)$ be a smooth connected compact boundaryless Riemannian surface.  For any $k \geq 2$, there exists a sequence of metrics $(g_n)_{n \geq 1}\in [g]$ of the form $ g_n=\mu_n g$ such that
$$
\lim_{n\to \infty} \lambda_k(g_n)=\tilde \Lambda_k(M, [g])
$$
 and  a probability measure $\mu$ such that
$$
\mu_{n}\rightharpoonup^*  \mu \,\,\,\,\text{weakly in measure as }  n\to +\infty. 
$$
 Moreover the following decomposition holds
\begin{equation}\label{decomp}
\mu=\mu_r+\mu_s
\end{equation}
where  $\mu_r $ is a $C^\infty$ nonnegative function and $\mu_s$ is the singular part given by 
$$
\mu_s=\sum_{i=1}^K c_i \delta_{x_i}
$$
for some $K \geq 1$, $c_i >0$ and some points $x_i \in M$. Furthermore, the number $K$ satisfies the bound 
$$
K \leq k-1
$$
Moreover, the weights $c_i >0$ belong to the discrete set 
\begin{equation}\label{weightSing}
c_i \in \bigcup_{j=1}^{k} \Big  \{ \frac{\tilde \Lambda_j(\mathbb S^2,[g_{round}])}{\tilde \Lambda_k(M, [g])} \Big \}. 
\end{equation}
The regular part of the limit density $\mu$, i.e. $\mu_r$ is either identically  zero or $\mu_r$ is absolutely continuous with respect to the riemannian measure with a  smooth positive density vanishing at most at a finite number of points on $M$. 

Furthermore, if we denote $A_r$ the volume of the regular part $\mu_r$, i.e. $A_r=A_{\mu_rg}(M)$, then $A_r$ belongs to the discrete set
\begin{equation}\label{weightReg}
A_r \in \bigcup_{j=0}^{k} \Big \{ \frac{\tilde \Lambda_j(M, [g])}{\tilde \Lambda_k(M, [g])} \Big \}.
\end{equation}

Finally, if we denote $\mathcal U$ the eigenspace of the Laplacien on $(M, \mu_rg)$ associated to the eigenvalue $\tilde \Lambda_k(M, [g])$, then there exists a family of eigenvectors $\left \{ u_1,\cdot \cdot \cdot,u_\ell \right \} \subset \mathcal U$ such that the map  

\begin{equation}
\left \{
\begin{array}{c}
\phi: M\to \mathbb R^\ell\\
x \to (u_1,\cdot \cdot \cdot,u_\ell)
\end{array} \right. 
\end{equation}  
is a minimizing harmonic map into the sphere $\mathbb S^{\ell-1}$.

\end{theorem}

Theorem \ref{main1} is a generalization of the result of our paper \cite{nS}. In \cite{nS} we proved  Theorem \ref{main1}
for $k=1$ under the assumption that $\tilde \Lambda_1(M, [g])>8\pi$. This last assumption was removed by Petrides in \cite{petrides}. He also suggested some modifications of our proof but basically following the same strategy.  Note that in case $k=1$ no bubbling phenomenon occurs, i.e. $\mu_s$ is identically zero. For
$k=2$ the bubbling phenomenon was observed on the sphere in \cite{nadirSphere}.

The proof of the previous theorem relies once again on a careful analysis of a Schr\"odinger type operator. Indeed consider $g' \in [g]$, by conformal invariance, the equation $-\Delta_{g'} u= \lambda_k (g')u$ reduces to the following problem
\begin{equation}\label{problem}
\left \{ 
\begin{array}{c}
-\Delta_g u =\lambda_k(g')\, \mu \,u ,\,\,\,\mbox{on}\,\,M\\
\int_M \mu \, dA_g=1. 
\end{array} \right . 
\end{equation}

\section{Proof of Theorem \ref{main1}: regular part of the extremal metric }

The proof follows our proof in \cite{nS}. We briefly sketch the main arguments. 

\subsection{Step 1: Regularization}
We perform a regularization by considering $S_n$ the class of densities $\mu $ such that $-\frac12\leq \mu \leq n$, $\int_M \mu dA_g=1$ for $n>0$. Denoting by $\lambda_k(\mu)$ the eigenvalue problem in \eqref{problem} with a density $\mu \in S_n$, we write 
$$\tilde \Lambda_n =\sup_{\mu \in S_n} \lambda_k(\mu). $$ 

We have by a direct application of well-known bounds for Schr\"odinger operators (see \cite{lieb}) that
\begin{pro}
For any given $n>0$, there exists a sequence $\left \{ \mu_{k,n} \right \}_{k \geq 0} \subset S_n$ such that as $k\rightarrow +\infty$

$$\mu_{k,n} \rightharpoonup^*  \mu_n \,\,\,\,\text{weakly in measure} $$

and 

$$\lambda_k(\mu_{k,n}) \rightarrow \tilde \Lambda_n .$$

Furthermore, we have  
 $$\int_M \mu_n\,dA_g=1$$
and 
$$ -\frac12 \leq \mu_n  \leq n.$$
\end{pro}

\subsection{Step 2: Passing to the limit in $n$ }

We need to control the two sets

$$E^n_{-}= \left \{ x \in M,\,\,\,-\frac12 \leq \mu_n(x) \leq 0 \right \}$$
and
$$E_n= \left \{ x \in M,\,\,\,\mu_n(x)=n \right \}.$$
 
 We have (see \cite{nS})
\begin{lemma}\label{measEn}
Let $n>0$. Then there exists $C>0$ such that  
$$A_g(E_n) \leq C/n. $$
\end{lemma}

We now come to the measure estimate of the set $E^n_{-}$.  
\begin{lemma} \label{negSet}
For any $n>0$, we have 
 $$A_g(E^n_-)=0.$$
\end{lemma}

\begin{remark}
The previous lemma and particularly its proof is instrumental in our paper with Grigor'yan \cite{GNS} on  bounds on the number of negative eigenvalues for Schr\"odinger operators with sign-changing potentials.   
\end{remark}

By the previous Lemmata, one can prove the following convergence result:

$$\mu_n \rightharpoonup^* \mu ,\,\,\,\text{weakly in measure as}\,\, n\to \infty$$

and furthermore

$$\mu >0 \,\,\mbox{a.e. in }M .$$

By Lebesgue decomposition theorem, we have 
\begin{equation}
\mu= \mu_r +\mu_s,
\end{equation}
where $\mu_s$, the singular part of the measure, can be decomposed into an absolutely singular part and a discrete part. 

%
%

The regularity theory developped in \cite{nS} (and invoking the result of Petrides \cite{petrides}) shows that the absolutely singular part of the measure $\mu$ vanishes identically. Hence the previous decomposition only involves a regular part, absolutely continuous with respect to the Riemannian measure, and a purely discrete singular part. 

Finally, by the results in \cite{nS}, the statements of Theorem \ref{main1} on the regular part follow. We refer the reader to our paper (see also \cite{petrides}). We postpone the proof of  \eqref{weightReg} to the next section since the proof is very similar to the one of \eqref{weightSing}.

\section{Proof of Theorem \ref{main1}: the singular part}

We decompose the spectrum of $-\Delta$ on $M$ endowed with the metric $\mu \,g$ into a regular part and a singular part. We give the definitions below. 

\begin{defi}\label{specreg}
We denote $\Lambda_r $ the discrete set of eigenvalues of $-\Delta$ on the manifold $(M,\mu_r \,g)$ where $\mu_r$ is the regular part of the measure $\mu$. 
\end{defi}
\begin{defi}\label{specsing}
We denote $\Lambda_s $ the finite discrete set of eigenvalues defined by: let $i=1,...,K$ where $K$ is the number of bubbles in Theorem \ref{main1}. Fix $k \geq 2$ and denote  
$$
\lambda^{x_i}_k=\lim_{\epsilon \to 0} \lambda^{x_i,\epsilon}_k 
$$ 
where $\lambda^{x_i,\epsilon}_k$ is a (nontrivial) eigenvalue in the flat metric for $\epsilon$ sufficiently small of 
\begin{equation}
\left \{ 
\begin{array}{c}
-\Delta u =\lambda u\,\,\,\,\mbox{on}\,\,B(x_i,\epsilon),\\
u=0\,\,\,\,\mbox{on}\,\,\partial B(x_i,\epsilon),
\end{array} \right . 
\end{equation}
Notice that the previous limit exists by monotonicity of the eigenvalues with respect to the domain. Then we have 
$$
\Lambda_s = \bigcup_{k \geq 0} \bigcup_{i=1}^K \lambda_k^{x_i},  
$$
and  $\lambda^{x_i}_k$ is a "singular" eigenvalue associated to the Dirac mass $\delta_{x_i}$ where $x_i \in M$. 
\end{defi}

We now split the spectrum in the following way
$$
Spec(-\Delta)= \Lambda_r \cup \Lambda_s. 
$$
By definition, if $\mu_r \equiv 0$ then $\Lambda_r$ is empty and the spectrum is purely singular. 

We first prove the following result on the singular spectrum. 

\begin{lemma}
If $\mu_r$ is not identically zero, then for any $i=1,...,K$ 
$$
\lambda^{x_i}_{K}=0,
$$
i.e. 
$$
K\leq k. 
$$
\end{lemma}
\begin{proof}
There exists $\delta >0$ such that one can find $K$ $\delta-$neighborhoods of each $x_i \in M$ such that they do not intersect. For any $\epsilon >0$, one constructs functions $\psi_{i,\epsilon}$ for any $i=1,...,K$ supported in a ball of center $x_i$ and radius $\delta$ with value $1$ in this ball and satisfying 
$$
\int_{B(x_i,\delta)} |\nabla \psi_{i,\epsilon}|^2 <\epsilon. 
$$
Then modifying the metric $g_n$ into $g_n'=\psi_{i,\epsilon} g_n$ such that 
$$
\lim_{n \to \infty} A_{g_n'}=c_i,
$$
one has for sufficiently large $n$ that 
$$
\frac{\int_{B(x_i,\delta)} |\nabla \psi_{i,\epsilon}|^2}{\int_{B(x_i,\delta)} |\psi_{i,\epsilon}|^2} <\frac{2\epsilon}{c_i}. 
$$
Since the functions $\psi_{i,\epsilon}$ have mutually disjoint supports, it follows that 
$$
\lambda^{x_i}_{K} < 2\frac{\epsilon}{\inf_i c_i},
$$
hence the result since $\epsilon$ can be taken arbitrary small. 
\end{proof}

As an improvement one has
\begin{lemma}
We have actually
$$
K \leq k-1. 
$$
\end{lemma} 
\begin{proof}
Assume by contradiction $K=k$. This implies that all the masses satisfy for $i=1,...,K$
$$
c_i=\frac{1}{K}
$$
and furthermore, 
$$
\mu_r \equiv  0. 
$$
Indeed if $\mu_r$ is not identically zero, then it implies by the previous lemma that $\lambda^{x_i}_K=0$. 
Denote $\bar \mu$ the measure on $M$ maximising the first non trivial eigenvalue in the conformal class $[g]$. By \cite{nS}, $\bar \mu$ has no singular part. As a consequence one has 
$$
\lambda_k =\frac{8\pi}{k}. 
$$
Assume that $M$ is not topologically a sphere. Then we have 
$$
\lambda_1(M,\bar \mu g) >8\pi. 
$$
We then modify the approximating sequence $g_n= \mu_n g$ into $\tilde g_n= \tilde \mu_n \,g$ sutch that 
$$
\tilde \mu_n \rightharpoonup^* \tilde \mu
$$ 
and 
$$
\tilde \mu= \tilde \mu_s +\tilde \mu_r
$$
where 
$$
\tilde \mu_r = \bar \mu /K
$$
and $\tilde \mu_s$ has point singularities at $x_1,...,x_{K-1}$ with weights $1/K$. Moreover one can choose the sequence $\tilde \mu_n$ such that the first eigenvalue at singular points will be  equal to $\frac{8\pi}{K}$. 

Since the first eigenvalue of the regular part $\tilde \mu_r$ is strictly larger than $8\pi/K$, it follows that 
$$
\lim_{n \to \infty} \lambda_K(M,\tilde \mu_n \,g) >\frac{8\pi}{K},
$$
hence a contradiction. 

Assume now that $M$ is a sphere. In that case, one can always assume without loss of generality, that the regular part $\mu_r$ is not identically zero. This can be done by composing the approximating measures $\mu_n$ with a suitable M\"obius transformation. Hence we are done.  
\end{proof}

We now prove the relations \eqref{weightSing} and \eqref{weightReg}. Actually the proofs are completely parallel and we just prove \eqref{weightSing}. This completes the proof of Theorem \ref{main1}. We then prove 

\begin{lemma}\label{weight}
Using the notations of Theorem \ref{main1}, one has 
\begin{equation}
c_i \in \bigcup_{j=1}^{k} \Big \{\frac{\tilde \Lambda_j(\mathbb S^2,[g_{round}])}{\tilde \Lambda_k(M, [g])} \Big \}. 
\end{equation}
\end{lemma}

The proof of Lemma \ref{weight} is done in several steps. We first have

\begin{lemma}\label{part1}
For any $i=1,...,K$, one has 
$$
\tilde \Lambda_k(M,[g]) \in \bigcup_{m=1}^\infty \{ \lambda_m^{x_i} \},
$$
where $\lambda_k^{x_i}$ are defined as above. 
\end{lemma}
\begin{proof}
Assume the contrary. Then there exists $i$ such that 
$$\tilde \Lambda_k(M,[g]) \notin  \bigcup_{m=1}^\infty \{ \lambda_m^{x_i} \}. $$
We then modify the metric in the following way: consider a smooth cut-off function for $\delta,\epsilon >0$ defined by
\begin{equation}
\psi(\delta,\epsilon)=
\left \{ 
\begin{array}{c}
1\,\,\,\mbox{on}\,\,M \backslash B(x_i,2\delta),\\
1-\epsilon\,\,\,\mbox{on}\,\,B(x_i,\delta),
\end{array} \right . 
\end{equation}
Then for sufficiently small $\epsilon$, there exists a sequence $\delta_n \to 0$ such that if we denote $\tilde g_n =\psi(\delta_n,\epsilon) g_n$ where $g_n$ is the metric constructed in the existence part of Theorem \ref{main1}, one has 
$$
\lim_{n \to \infty} \lambda_k(\tilde g_n)=\tilde \Lambda_k(M,[g]).
$$
By properly choosing the sequence $\delta_n$, one has 
$$
A_{\tilde g_n}(M) \leq \frac12 c_i \epsilon. 
$$
Therefore, the normalized metric 
$$
h_n= \frac{\tilde g_n}{A_{\tilde g_n}(M)} \in [g_n]
$$
is such that 
$$
\lim_{n \to \infty} \lambda_k(h_n)>\tilde \Lambda_k(M,[g]),
$$
hence contradicting the definition of $\tilde \Lambda_k(M,[g])$. 

\end{proof}

One has similarly 
\begin{lemma}\label{part2}
For any $i=1,...,K$, one has 
$$
\tilde \Lambda_k(M,[g]) \in \Lambda_r. 
$$
\end{lemma}

We are then in position to prove Lemma \ref{weight}. 

\subsection*{Proof of Lemma \ref{weight}} 
Assume by contradiction that there exists an index $i$ such that 
\begin{equation}
c_i \notin \bigcup_{j=1}^{k-1} \Big \{\frac{\Lambda_j(\mathbb S^2,[g_{round}])}{\tilde \Lambda_k(M, [g])} \Big \}. 
\end{equation}

By Lemmata \ref{part1} and \ref{part2}, there exists an $m \geq 1$ such that 
$$
\lambda_m^{x_i}=\tilde \Lambda_k(M,[g])
$$

For $\epsilon >0$ and $i=1,...,K$, we introduce a sequence of metrics on $\mathbb S^2$. Let $\Omega \subset \mathbb S^2$ be an open domain and $\psi: \Omega \to M$ a conformal map such that $x_i \in \psi(\Omega)$. Set 
$$
\tilde g^{\epsilon,i}_n=\psi^*g'_n|_{B(x_i,\epsilon)},
$$
such that $A_{\tilde g^{\epsilon,i}_n} \to c_i$. 
If $\lambda_j^{\epsilon,n,i}$ denote the eigenvalue of the Laplace-Beltrami operator on $(\mathbb S^2,\tilde g^{\epsilon, i}_n)$ then one can define the following limits: 
$$
\tilde \lambda_j^{\epsilon,i}= \lim_{n \to \infty}\lambda_j^{\epsilon,n,i},
$$ 
$$
\tilde \lambda_j^{i}= \lim_{\epsilon \to 0}\tilde \lambda_j^{\epsilon,i}
$$
we prove
$$
\tilde \lambda_m^{i}=\lambda^{x_i}_m=\frac{\Lambda_m(\mathbb S^2,[g_{round}])}{c_i}. 
$$
Assume not, i.e.,
$$
\lambda^{x_i}_m\neq\frac{\Lambda_m(\mathbb S^2,[g_{round}])}{c_i}
$$
Then we can modify metrics $\tilde g^{\epsilon,i}_n$ into the metrics $\tilde{\tilde g}^{\epsilon,i}_n$ on the sphere
having the same area such that the area of $\mathbb S^2\setminus \Omega$ in the new metrics tends to $0$ and the corresponding eigenvalues $\Lambda_j^{\epsilon,n,i}$ satisfy
$$
\lim_{\frac1n, \epsilon \to 0}\Lambda_j^{\epsilon,n,i}>\tilde \lambda_m^{i}=\lambda^{x_i}_m
$$
One can modify the sequence $g_n$, the one obtained in the existence part of Theorem \ref{main1}, in a neighborhood of $x_i$ into a metric $g'_n$ by transplanting the problem to the sphere $\mathbb S^2$ in the following way.  We set on $\Omega$
$$
\tilde{\tilde g}^{\epsilon,i}_n=\psi^*g'_n|_{B(x_i,\epsilon)},
$$
Then we have 
$$
\lim_{n \to \infty} \lambda_k(g'_n)>\tilde \Lambda_k(M,[g]). 
$$
On the other hand, 
$$
\tilde \Lambda_k (M,[g]) \notin \bigcup_{m=1}^\infty \left \{ \lambda_m^{x_i} \right \},
$$
hence a contradiction. This finishes the proof of Theorem \ref{main1}.

\bibliographystyle{alpha}
\bibliography{biblio}

\medskip

{\em NN} -- 
CNRS, I2M UMR 7353-- Centre de Math\'ematiques et Informatique, Marseille, France. 
 
{\tt nicolas@cmi.univ-mrs.fr}

\medskip

{\em YS} --  
Universit\'e Aix-Marseille, I2M UMR 7353-- Centre de Math\'ematiques et Informatique, Marseille, France. 

{\tt yannick.sire@univ-amu.fr} 

\end{document}